\documentclass[11pt]{amsart}
\usepackage[cmtip,all]{xy}
\usepackage{amssymb}
\newtheorem{theorem}{Theorem}[section]
\newtheorem{lemma}[theorem]{Lemma}

\theoremstyle{definition}
\newtheorem{definition}[theorem]{Definition}

\newenvironment{claim}[1]{\par\noindent\underline{Claim:}\space#1}{}
\newenvironment{claimproof}[1]{\par\noindent\underline{Proof:}\space#1}{\hfill $\blacksquare$}

\theoremstyle{remark}

\numberwithin{equation}{section}

\begin{document}

\title{A vanishing theorem for log canonical pairs after de Fernex-Ein }

\author{Chih-Chi Chou}
\address{University of Illinois at Chicago, Department of Mathematics, Statistics, and Computer Science, Chicago IL 60607  }
\address{cchou20@uic.edu}

\begin{abstract}
We extend the main vanishing theorem in de Fernex-Ein \cite{DE} to singular varieties without assuming locally complete intersection.
\end{abstract}

\maketitle

\section{Introduction}

{\let\thefootnote\relax\footnote{{\it 2010 Mathematics Subject Classification}: 14B05, 14F17}}

In this note we prove the following Nadel vanishing type theorem,

 \begin{theorem}\label{nadel1}
Given a  log canonical pair $(X, \Delta ;  eZ)$, where $Z\subset X$ is a pure-dimensional reduced subscheme of codimension $e$. 
Suppose that none of the components of $Z$  is  contained in $Sing(X)\cup  Supp(\Delta)$,
then for any nef line bundles $A$ and $M$, such that $A\otimes \mathcal{O}(-K_X-\Delta)$ is ample and $M\otimes \mathcal{I}^{\otimes e}_Z$ is globally generated, we have 
\begin{equation*}
H^i(X, A\otimes M \otimes \mathcal{I}_Z)=0  \hspace{2mm} \mbox{for} \hspace{2mm}  i>0.
\end{equation*}
 \end{theorem}

In particular, suppose $Z$ is scheme-theoretically given by
\begin{equation*}
Z=H_1\cap\cdots \cap H_t
\end{equation*}
for some divisors $H_i\in \lvert  \mathcal{ L}^{\otimes d_i}\rvert$, where $ \mathcal{ L}$ is a globally generated line bundle 
such that $d_1\ge \cdots \ge d_t$.
Then
\begin{equation*}
H^i(X, A\otimes \mathcal{L}^{\otimes k}\otimes \mathcal{I}_Z)=0 \hspace{2mm} \mbox{for} \hspace{2mm} i>0, k\ge e\cdot d_1. 
\end{equation*}
This result partially\footnote{The lower bound of $k$ is bigger than the one in \cite{DE}.} 
 generalizes the main vanishing theorem in de Fernex-Ein \cite{DE}, which assumes $X$ is a locally complete intersection variety with rational singularities.

It has been some efforts generalizing  the vanishing theorem in \cite{DE}.
In \cite{Niu11} Niu proves an  analogous vanishing theorem for power of ideal sheaves.
In \cite{Niu13}, using technique of generic linkage, the vanishing theorem in \cite{DE} is generalized to 
pairs $(\mathbb{P}^n,eZ )$ which are log canonical except at finitely many points.   
In this note, we also prove this type of generalization,
 \begin{theorem}\label{thm2}
The conclusion of Theorem \ref{nadel1} is still true if  $(X, \Delta; eZ)$ is  log canonical except at finitely many points. 
\end{theorem}

In \cite{DE} the main application is questions related to Castelnuovo-Mumford regularity of singular subvarieties in projective spaces,
which generalizes the results in Bertram-Ein-Larzarsfeld \cite{BEL} and Chardin-Ulrich \cite{CU}. (see also Niu \cite{Niu}, \cite{Niu12} for related results).

On the contrary,
the main results of this note are in particular suitable for regularity problems of subvarieties in singular varieties.
For example, given a curve $Z$ in a Schubert variety $X$.
By result of \cite{AS}, there is some $\Delta$   such that $(X, \Delta)$ is log terminal.
and the support of $\Delta$ is in the complement of   the biggest Schubert cell $U\subset  X$ (a dense smooth open set).
So suppose the generic point of $Z$ is contained in $U$, then 
$(X; \Delta, eZ)$ is  log canonical except at finitely many points. 
Then by Theorem \ref{thm2} we have a bound of regularity in terms of the degrees of generators of $\mathcal{I}_Z$.

We make a remark  about the  the proof in this paper .
The assumption of locally complete intersection is crucial to \cite{DE} because their strategy is to approximate $\mathcal{I}_Z$ by some multiplier ideal sheaf.
To achieve this they need to use  inversion of adjunction theorem of locally complete intersection varieties.
 
In this paper , we apply a cohomology tool for log canonical pairs developed by Ambro-Fujino.
We consider $\hat{X}$, the normalization of blowing up of $X$ along $Z$.
The  idea is that when we pull back $A\otimes M\otimes \mathcal{I}_Z$ to $\hat{X}$, it becomes a line bundle and has more 
positivity properties. (Note that $\mathcal{I}_Z\cdot \mathcal{O}_{\hat{X}}$ is relative ample.) 
Using Ambro-Fujino's theorem (Theorem \ref{fujino}) we are able to reduce the vanishing result on $X$ to $\hat{X}$.

The organization of this paper goes as follows. In the second section we recall some notions we need , including the vanishing theorem by Ambro-Fujino.
The proofs of the main theorems are in the last section. \\

{\it Acknowledgements.}
I would like to thank Izzet Coskun, Lawrence Ein and Mihnea Popa for many useful 
discussions.
I am also grateful to Wenbo Niu for references and pointing out Theorem \ref{thm2}.

\section{preliminaries}

We first recall the notion of singularities of pair.
We say $(X,\Delta)$ is a pair if $X$ is normal and $K_X+\Delta $ is $\mathbb{Q}$-Gorenstein.
More generally, we consider $(X, \Delta; eZ)$ where $Z$ is any subscheme of $X$ and $e$ is a positive integer.
Take a log resolution  $f:Y\rightarrow X$ such that the support of $\mbox{Exc(f)}\cup f^{-1}_*\Delta \cup \mathcal{I}_Z\cdot \mathcal{O}_Y$ is a union of simple normal crossing divisors.

\begin{definition}
Given a resolution $f:Y\rightarrow X$ as above, suppose $\mathcal{I}_Z\cdot \mathcal{O}_Y=\mathcal{O}_Y(-F)$,
we let $ div( \mathcal{I}_Z\cdot \mathcal{O}_Y)$ denote $-F$. In particular, $e\cdot div( \mathcal{I}_Z\cdot \mathcal{O}_Y)=-eF.$
\end{definition}
A pair $(X, \Delta; eZ)$ is called log canonical if for any log resolution as above  we can write
\begin{equation*}
K_Y-f^*(K_X+\Delta)+e\cdot div( \mathcal{I}_Z\cdot \mathcal{O}_Y)=P-N
\end{equation*}
where both $P$ and $N$ are effective, without common components and all coefficients of components of $N$  are less or equal to one.

The following theorem due to Ambro-Fujino is essential to the proof of  theorem \ref{nadel1}. See \cite{Al} and \cite{OF} Theorem 6.3.
\begin{theorem}\label{fujino}
Let $(Y, \Delta)$ be a simple normal
crossing pair such that $\Delta$ is a boundary $\mathbb{R}$-divisor on Y . Let $f : Y \rightarrow X$
be a proper morphism between algebraic varieties and let $L$ be a Cartier
divisor on $Y$ such that $L - (K_Y + \Delta)$ is $f$-semi-ample.
\begin{enumerate}
\item every associated prime of $R^qf_*O_Y (L)$ is the generic point of the

$f$-image of some stratum of $(Y, \Delta)$.
\item let $\pi : X \rightarrow V$ be a projective morphism to an algebraic variety
$V$ such that
\begin{equation*}
L - (K_Y + \Delta) \sim _{\mathbb{R}} f^*H
\end{equation*}
for some $\pi$-ample $\mathbb{R}$-Cartier $\mathbb{R}$-divisor H on X. Then $R^qf∗\mathcal{O}_Y (L)$
is  $\pi _∗$-acyclic, that is,
\begin{equation*}
R^p \pi_*R^qf_*\mathcal{O}_Y(L)=0
\end{equation*}
for every $p > 0$ and $q \ge 0$.
\end{enumerate}
\end{theorem}

\section{Proofs of main theorems}

Through out this section
we consider a pair $(X, \Delta; eZ)$, where $Z\subset X$ is a pure-dimensional reduced subscheme of codimension $e$.
We also require that all of  the irreducible components of $Z$ are not contained in either sing$(X)$ or supp($\Delta$).
We first prove the following lemma needed later.

\begin{lemma}\label{resolution}
There is a log resolution $f:Y\rightarrow (X, \Delta; eZ)$ such that,
\begin{equation*}
K_{Y}-f^*(K_X+\Delta)+e\cdot \mbox{div}(\mathcal{I}_Z\cdot \mathcal{O}_Y)=P-N-E_Z,
\end{equation*}
  where both $P$ and $N$ are effective,  and $E_Z$ is the sum of all components of the support of $\mathcal{I}_Z\cdot \mathcal{O}_Y$ mapping to the generic points of  $Z$.
And the support of $P$ and $N$ does not contain any divisor mapping to any generic points of  $Z$.
In particular, $f(E_Z)=Z$.
\end{lemma}

\begin{proof}

Since no component of $Z$ is contained in  $ \mbox{sing}(X)$  , 
we can take factorizing resolution of $Z$ in $X$.
Moreover, since  no component of $Z$ is contained in  supp$( \Delta)$  ,
by Corollary 3.2 in \cite{EE} we can take this resolution to be a log resolution of $(X, \Delta)$.
So we have a log resolution $f:Y_1\rightarrow (X, \Delta)$ such that\\
(i) $f_1$ is an isomorphism over the generic points of $Z$.\\
(ii)$\mathcal{I}_{Z}\cdot \mathcal{O}_{Y_1}=\mathcal{I}_{Z}\cdot \mathcal{L}$, where $\mathcal{L}$ is a line bundle.\\
(iii)The strict transform $Z_1\subset Y_1$ of $Z$ is smooth and $Z_1\cup\mbox{exc}(f_1)\cup \mbox{supp}(f_1^{-1}\Delta)$ is simple normal crossing.

Next we take the blow up of $Y_1$ along $Z_1$ and denote it by $f_2:Y_2\rightarrow Y_1$, and let $Z_2$ be the exceptional divisor of $f_2$ . In summary, we have the following diagram.
  \begin{displaymath}
    \xymatrix{ Z_2\ar[d]\ar@{^{(}->}[r] & Y_2 \ar[d]^{f_2} \\
               Z_1\ar[d]\ar@{^{(}->}[r]& Y_1\ar[d]^{f_1}\\
               Z \ar@{^{(}->}[r]&X }
\end{displaymath}

Let $K_{Y_1}-f_1^*(K_X+\Delta)=\sum a_iE_i$.
Since $f_1$ is an isomorphism over the generic points of $Z$, none of the $E_i$'s is mapping to generic points of  $Z$. 
On the other hand, $f_2$ is an blowing up of smooth variety along smooth subscheme of codimension $e$, we have
$K_{Y_2}-f_2^*K_{Y_1}=(e-1)Z_2$.

Let $\mathcal{I}_Z\cdot \mathcal{O}_{Y_2}=\mathcal{O}_{Y_2}(-Z_2-F)$ where $Z_2$ is mapped to generic points of $Z$ and $F$ is mapped to non-generic points.
Then we have 
\begin{align*}
\hspace{.5in}& K_{Y}-f^*(K_X+\Delta)+e\cdot div(\mathcal{I}_Z\cdot \mathcal{O}_{Y_2})\\
&=K_{Y_2}-f_2^*K_{Y_1}+f_2^*(K_{Y_1}-f_1^*(K_X+\Delta))\underbrace{-e\cdot Z_2-e\cdot F}_\text{$e\cdot div(\mathcal{I}_Z\cdot \mathcal{O}_{Y_2})$}\\
&=(e-1)\cdot Z_2+f_2^*(\sum a_iE_i)-e\cdot Z_2-e\cdot F\\
&:=P-N-Z_2
\end{align*}

Since none of the $E_i$'s  passes through the generic point of $Z_1$, the support of $f_2^*(\sum a_iE_i)$
 does not contain any component of  $Z_2$.
In conclusion,  the support of $P$ and $N$  does not contain any divisor dominating any component of $Z$.
\end{proof}

\begin{theorem}(=Theorem \ref{nadel1})\label{nadel}
Suppose the pair $(X, \Delta; eZ)$ is  log canonical. 
Then for any nef  line bundles $A$ and $M$, such that $A\otimes \mathcal{O}(-K_X-\Delta)$ is ample and $M\otimes \mathcal{I}^{\otimes e}_Z$ is globally generated, we have 
\begin{equation*}
\mbox{H}^i(X, A\otimes M \otimes \mathcal{I}_Z)=0  \hspace{2mm} \mbox{for} \hspace{1mm} i>0.
\end{equation*}

\end{theorem}

\begin{proof}
Let $h:\hat{X}\rightarrow X$ be the normalization of blowing up of $X$ along $Z$, and line bundle $\mathcal{O}_{\hat{X}}(1)= \mathcal{I}_Z\cdot \mathcal{O}_{\hat{X}}$. Note that $\mathcal{O}_{\hat{X}}(1)$ is $h$-ample and $h_*\mathcal{O}_{\hat{X}}(1)=\bar {\mathcal{I}}_Z=\mathcal{I}_Z$ since $\mathcal{I}_Z$ is a radical ideal.

Take a log resolution $f:Y\rightarrow (X,\Delta; eZ)$ as in lemma \ref{resolution}.
Since the inverse image $\mathcal{I}_Z\cdot \mathcal{O}_Y$ is invertible and $Y$ is smooth, we see that $f$
 factors through $h$, with morphism $g:Y \rightarrow \hat{X}$.
 We also note that  $g^*\mathcal{O}_{\hat{X}}(1)=\mathcal{I}_Z\cdot \mathcal{O}_Y$.
Then by the log canonical assumption we have 
\begin{equation}\label{equation}
K_Y-f^*(K_X+\Delta)+e\cdot \mbox{div}(g^*\mathcal{O}_{\hat{X}}(1))=P-B-E-E_Z,
\end{equation}
where $P$ is effective, $\lfloor B\rfloor =0 $, $E$ is reduced and $E_Z$ is the components of the support of $\mathcal{I}_Z\cdot \mathcal{O}_Y$ mapping to the generic points  of  $Z$.  
Note that the support of $P$ does not dominate any component of $Z$ by lemma \ref{resolution}.

To apply theorem \ref{fujino}, we go as following.
First we define $\Delta _Y=\lceil P\rceil-P+B+E $, then $(Y, \Delta_Y)$ is a simple normal crossing pair.
Then we define $L$ by adding $f^*A+f^*M+\Delta_Y$ to both hand sides of equation (\ref{equation}), that is,
\begin{equation*}
L:=f^*A+f^*M+K_Y-f^*(K_X+\Delta)+e\cdot \mbox{div}(g^*\mathcal{O}_{\hat{X}}(1))+\Delta_Y=f^*A+f^*M-E_Z+\lceil P\rceil
\end{equation*}
Note that then $L$ is a line bundle.

To apply Theorem \ref{fujino} (2), we calculate
\begin{align*}
 L-(K_Y+\Delta_Y) &= f^*A-f^*(K_X+\Delta)+f^*M+e\cdot \mbox{div}(g^*\mathcal{O}_{\hat{X}}(1)) \\
 &= g^*(\underbrace{h^*A-h^*(K_X+\Delta)}_\text{semiample}+\underbrace{h^*M+e\cdot \mathcal{O}_{\hat{X}}(1)}_\text{globally genarated})\\
 &:= g^*H
\end{align*}
because by assumption $A-K_X$ is ample and $M\otimes \mathcal{I}^{\otimes e}_Z$ is globally generated.
In particular, $H$ is semiample.
Plus the fact that $H.C>0$ for every curve $C$ on $\hat{X}$ (easy to check),
we conclude that $H$ is in fact an ample line bundle on $\hat{X}$.
Then by Theorem \ref{fujino} (2), we have $H^i(\hat{X}, g_*L)=0  \hspace{2mm}  \mbox{for} \hspace{1mm}  i>0.$
Moreover, $H$ is also $h$-ample. 
So by loc. cit. we have $R^jh_* g_*L=0, \forall j>0.$ 
As a result,
\begin{equation}\label{cohomology}
0=H^i(\hat{X}, g_*L)=H^i(X,h_* g_*L ), \forall i > 0.
\end{equation}
The theorem follows from this equation and the next claim.
\begin{claim}
 $h_*g_*L=f_*L= A\otimes M \otimes \mathcal{I}_Z$.\\
\end{claim}
\begin{claimproof} Since $L=f^*A+f^*M-E_Z+\lceil P\rceil$, by projection formula it suffices to prove 
\begin{equation}
h_*\mathcal{O}_{\hat{X}}(1)=h_*(g_*\mathcal{O}_Y(\lceil P\rceil-E_Z)).
\end{equation}
To this aim, note that 
\begin{equation*}
\mathcal{I}_Z=h_*\mathcal{O}_{\hat{X}}(1)\subset h_*g_*(\mathcal{O}_Y(\lceil P\rceil)+g^*\mathcal{O}_{\hat{X}}(1))\subset h_*g_*\mathcal{O}_Y(\lceil P\rceil -E_Z) \subset \mathcal{O}_X
\end{equation*}
In other words, $h_*g_*\mathcal{O}_Y(\lceil P\rceil-E_Z)$ is an ideal sheaf on $X$.
By Lemma \ref{resolution}, $\lceil P\rceil$ and $E_Z$ do not have common components and $\lceil P\rceil$ is an effective exceptional divisor, so 
\begin{equation*}
h_*g_*\mathcal{O}_Y(\lceil P\rceil-E_Z)=f_*\mathcal{O}_Y(\lceil P\rceil-E_Z)=f_*\mathcal{O}_Y(-E_Z).
\end{equation*}
Note that  $f_*\mathcal{O}_Y(-E_Z)$ is an ideal sheaf determines a scheme set theoretically equal to $Z$.
So we have
\begin{equation*}
\mathcal{I}_Z\subset f_*\mathcal{O}_Y(-E_Z)\subset \sqrt{\mathcal{I}_Z}.
\end{equation*}
 Hence we have equation (3.3) and  the claim.
\end{claimproof}

\end{proof}

With same notations as Theorem \ref{nadel},
\begin{theorem}(=Theorem \ref{thm2})
The conclusion of Theorem \ref{nadel} is still true if  $(X, \Delta; eZ)$ is  log canonical except at finitely many points. 
\end{theorem}

\begin{proof}
Apply Lemma \ref{resolution}, we have the following equation similar  to equation (\ref{resolution}),
\begin{equation*}
K_{Y}-f^*(K_X+\Delta)+e\cdot \mbox{div}(\mathcal{I}_Z\cdot \mathcal{O}_Y)=P-N-E_Z-F.
\end{equation*} 
The only difference is that there is a term $-F$ denoting the exceptional over the 
non-log canonical points.
We remark that by the assumption $F$ is non-reduced.  

Let $\mathcal{I}^*=f_*\mathcal{O}_Y(-E_Z-F)$.
Then exactly the same argument as in the proof of Theorem \ref{nadel} shows
\begin{equation*}
H^i(X, A\otimes M \otimes \mathcal{I}^*)=0, \forall i>0.
\end{equation*}

On the other hand we have,
\begin{equation*}
0\rightarrow \mathcal{I}^* \rightarrow \mathcal{I}_Z \rightarrow \mathcal{Q} \rightarrow 0,
\end{equation*}
where $\mathcal{Q}$ is a sheaf supported on some close points by assumption.
Take the  long exact sequence we have 
\begin{equation*}
H^i(X, A\otimes M \otimes \mathcal{I}_Z)=0, \forall i>0.
\end{equation*}
\end{proof}

\bibliographystyle{amsalpha}

\end{document}